\documentclass[12pt]{amsart}
\usepackage{amstext,amsfonts,amssymb,amscd,amsbsy,amsmath,verbatim,fullpage}
\usepackage[alphabetic,lite,backrefs]{amsrefs} 
\usepackage[bookmarks, colorlinks=true, linkcolor=blue, citecolor=blue, urlcolor=blue, pagebackref=true]{hyperref}
\usepackage{ifthen,multicol, extarrows,pgfplots}
\usepackage{color,tikz}
\usepackage{amsthm}
\usepackage{latexsym}
\usepackage[all]{xy}
\usepackage{enumerate}
\usepackage{eucal}

\numberwithin{equation}{section}
\newtheorem{lemma}[equation]{Lemma}
\newtheorem{theorem}[equation]{Theorem}
\newtheorem{prop}[equation]{Proposition}
\newtheorem{cor}[equation]{Corollary}
\newtheorem{conj}[equation]{Conjecture}

\newtheorem{claim*}{Claim}
\newtheorem{thm}[equation]{Theorem}

\newtheorem{question}[equation]{Question}

\theoremstyle{definition}
\newtheorem{rmk}[equation]{Remark}
\newenvironment{remark}[1][]{\begin{rmk}[#1] \pushQED{\qed}}{\popQED \end{rmk}}
\newtheorem{eg}[equation]{Example}
\newenvironment{example}[1][]{\begin{eg}[#1] \pushQED{\qed}}{\popQED \end{eg}}
\newtheorem{defn}[equation]{Definition}

\newcommand{\codim}{\operatorname{codim}}

\newcommand{\rB}{\mathrm{B}}
\newcommand{\rX}{\mathrm{X}}
\newcommand{\rY}{\mathrm{Y}}
\newcommand{\rZ}{\mathrm{Z}}

\newcommand{\PP}{\mathbb P}

\newcommand{\ZZ}{\mathbb Z}

\newcommand{\rE}{\mathrm E}
\newcommand{\rC}{\mathrm C}

\newcommand{\bE}{\mathbf E}

\newcommand{\bP}{\mathbf P}
\newcommand{\bk}{\mathbf{k}}

\newcommand{\cO}{\mathcal O}

\newcommand{\HH}{\widetilde{\mathrm H}}

\def\reg{\operatorname{reg}}

\DeclareMathOperator{\link}{link}

\newcommand{\oct}{\diamondsuit}
\DeclareMathOperator{\pdim}{pdim}
\DeclareMathOperator{\Var}{Var}

\title{Random Flag Complexes and Asymptotic Syzygies}
\author{Daniel Erman}
\address{Daniel Erman: Department of Mathematics, University of Wisconsin,
  Madison, Wisconsin, 53706, United States of America; {\normalfont
    \texttt{derman@math.wisc.edu}}}

\author{Jay Yang}
\address{Jay Yang: Department of Mathematics, University of Wisconsin,
  Madison, Wisconsin, 53706, United States of America; {\normalfont
    \texttt{yangjay@math.wisc.edu}}}
\thanks{The authors were  supported by NSF grants DMS-1601619 and DMS-1502553.}

\begin{document}

\begin{abstract}
We use the probabilistic method to construct examples of conjectured phenomena about asymptotic syzygies.  In particular, we use Stanley-Reisner ideals of random flag complexes to construct new examples of Ein and Lazarsfeld's nonvanishing for asymptotic syzygies and of Ein, Erman, and Lazarsfeld's conjecture on how asymptotic Betti numbers behave like binomial coefficients.
\end{abstract}

\maketitle

\setcounter{section}{1}

Using the probabilistic method, we produce examples of conjectured behavior on asymptotic syzygies.  One of these provides the first known example of a phenomenon conjectured by Ein, Erman, and Lazarsfeld.  

Our central construction involves random flag complexes.  We use $G\sim G(n,p)$ to denote an Erd\H{o}s-R\'enyi random graph on $n$ vertices, where each edge is attached with probability $p$.  We turn $G$ into a flag complex by adjoining a $k$-simplex to every $(k+1)$-clique in the graph, and $\Delta\sim\Delta(n,p)$ denotes a flag complex chosen with respect to this distribution.  The properties of random flag complexes have been studied extensively in recent years; see~\cite{kahle-survey} for a survey of recent results.  From $\Delta$, Stanley-Reisner theory yields a squarefree monomial ideal $I_\Delta\subseteq \bk[x_1,x_2,\dots,x_n]$~\cite[Chapter~5]{bruns-herzog}, and we analyze the Betti numbers of $I_\Delta$.

A recent paper De Loera, Petrov\'ic, Silverstein, Stasi, and Wilburne~\cite{dpssw} also produces random monomial ideals via a construction similar to Erd\H{o}s-R\'enyi random graphs, and one of their constructions specializes to ours.  They study thresholds and the distribution of algebraic invariants in this framework, and they provide an array of results and conjectures. 

We are motivated by questions and conjectures about asymptotic syzygies.  These questions are generally outside of the range computable in Macaulay2~\cite{M2} or elsewhere, and so there is a lack of known examples.  By contrast, results on random flag complexes are asymptotic in nature.  By using probabilistic techniques to analyze the syzygies of $I_\Delta$, we produce new examples of behaviors conjectured in~\cite{ein-laz-asymptotic} and ~\cite{ein-erman-laz-random}.

We now summarize Ein and Lazarsfeld's central result on asymptotic syzygies.  For a graded module $M$ over a polynomial ring, we recall that $\beta_{i,j}(M)$ denotes the number of minimal generators of degree $j$ of the $i$th syzygy module of $M$; see~\cite[\S1B]{eisenbud-syzygies} for a review.  We define $\rho_k(M)$ as the ratio of nonzero entries in the $k$th row of the Betti table:
\[
\rho_k(M) := \frac{\# \{ i \in [0,\pdim(M)] \text{ where }  \beta_{i,i+k}(M)\ne 0\}}{ \pdim(M)+1}.
\]
Under increasingly positive embeddings, ~\cite{ein-laz-asymptotic} shows that these densitites approach $1$.
\begin{thm}[Ein-Lazarsfeld, 2012]\label{thm:EL}
Let $X$ be a smooth, $d$-dimensional projective variety and let $A$ be a very ample divisor on $X$.  For any $n\geq 1$, let $S_n$ be the homogeneous coordinate ring of $X$ embedded by $nA$. For each $1\leq k \leq d$, $\rho_k(S_n)\to 1$ as $n\to \infty$.
\end{thm}
\noindent See~\cite[Theorem~A]{ein-laz-asymptotic} for the sharper result and Figure~\ref{fig:ein-laz} for an illustration.
A similar nonvanishing phenomenon was shown to hold for integral varieties~\cite[Theorem, p.~2256]{zhou}, arithmetically Cohen-Macaulay varieties~\cite[Theorem~3.1]{ein-erman-laz-quick}, and certain iterated subdivisions of Stanley-Reisner rings~\cite{cjkw}.  Moreover, experiments in Macaulay2 with different asymptotic families of ideals (graph curves, unions of linear spaces, etc.) suggest that this asymptotic nonvanishing behavior occurs in a broad range of examples.  This motivates the following question.

\begin{figure}
\begin{tikzpicture}[yscale=0.6]
 \node[fill,circle,inner sep = 0.01cm] at (0,1) {};
 \foreach \i in {1,...,221} {
   \node[fill,circle,inner sep = 0.01cm] at (\i*0.05cm,0) {};
 }
 \foreach \i in {28,...,275} {
   \node[fill,circle,inner sep = 0.01cm] at (\i*0.05cm,-1cm) {};
   }

 \foreach \i in {199,...,282} {
   \node[fill,circle,inner sep = 0.01cm] at (\i*0.05cm,-2cm) {};
 }

\end{tikzpicture}
\caption{Each dot represents a known nonzero entry in the Betti table of $\PP^3$ embedded by $\cO(n)$ for $n=10$.  By Ein and Lazarsfeld's Theorem~\ref{thm:EL}, the density of the dots in rows $1,2,$ and $3$ will approach $1$ as $n\to \infty$.  Theorem~\ref{thm:nonvanishing} shows a similar phenomenon holds for ideals of random flag complexes.}
\label{fig:ein-laz}
\end{figure}

\begin{question}\label{q:nonvanishing density}
Let $\{I_n\}$ be a family of ideals where $\pdim(I_n)\to \infty$.  Fix some $k$.  Under what conditions will $\rho_k(S/I_n)\to 1$ as $n\to \infty$?
\end{question}
One way to understand these asymptotic nonvanishing results is by considering the overlaps between the nonzero entries in the rows of the Betti table.  The Hilbert function of a graded module will determine the alternating sum of the entries along the slope one diagonals of the Betti table. We define {\bf overlapping Betti numbers} as Betti numbers that are not determined by the Hilbert function: e.g. when $\beta_{i,j}$ and $\beta_{i+1,j}$ are both nonzero.  Theorem~\ref{thm:EL} and the related followup results show that such overlapping Betti numbers are the norm in many different families of examples.

While Question~\ref{q:nonvanishing density} addresses qualitative expectations about asymptotic syzygies, the corresponding quantitative behavior of asymptotic syzygies was raised in~\cite{ein-erman-laz-random}.  They introduce a random Betti table model to provide a heuristic for the asymptotic behavior of certain families of Betti tables.  Their analysis suggests that, roughly speaking, each row of the Betti table of any very positive embedding displays the pattern of a large Koszul complex~\cite[Conjecture~B and Theorem~C]{ein-erman-laz-random}.  
Yet despite the expectation that this behavior should be common, the only known occurrence is for a smooth curve of high degree~\cite[Proposition~A]{ein-erman-laz-random}.  

\medskip

Our main results provide new families whose Betti tables exhibit the conjectured behaviors described above.  We write $f(n) \ll g(n)$ if $\lim_{n\to \infty} \frac{f(n)}{g(n)}=0$.

\begin{theorem}\label{thm:nonvanishing}
Fix some $r\geq 1$.  Let $\Delta \sim \Delta(n,p)$ with $\frac{1}{n^{1/r}}\ll p \ll 1$.  For each $1\leq k \leq r+1$, we have $\rho_k(S/I_\Delta)\to 1$ in probability.
\end{theorem}
Saying that $\rho_k(S/I_\Delta)\to 1$ in probability is equivalent to asking that for any $\epsilon>0$, the probability that $\rho_k(S/I_\Delta)\geq 1 -\epsilon$ goes to $1$ as $n\to \infty$.  In particular, 
for the given parameter range, random flag complexes in the $\Delta(n,p)$
model provide a positive answer to Question~\ref{q:nonvanishing density}, similar to Theorem~\ref{thm:EL}.  See Example~\ref{ex:nonvanishing}.

The proof of Theorem~\ref{thm:nonvanishing} uses randomness to find particular subcomplexes of $\Delta$.  As we will review in Section~\ref{sec:background}, the boundary complex of the $(s+1)$-dimensional octahedron has the minimal number of edges possible for a flag complex with $(s+1)$th homology, and it is thus the most likely subcomplex to contribute to the $(s+1)$th row of the Betti table of $S/I_\Delta$.  The main step of the proof comes from Theorem~\ref{thm:threshold} below, where we show that the bound $\frac{1}{n^{1/s}}\ll p$ is the threshold for the existence of this particular subcomplex.  Once we have crossed this threshold, we can find this particular subcomplex, and minor variants of it, yielding nonzero Betti numbers throughout nearly the entire $(s+1)$th row.

Next we construct examples whose Betti tables exhibit the more detailed asymptotics suggested in~\cite{ein-erman-laz-random}. For any $I_\Delta$, the Hilbert function of $S/I_\Delta$ will have the form $(1,n,\dots)$, and thus as $n\to \infty$, the Betti table will necessarily scale with $n$.  To account for this growth, we normalize the Betti table, defining $\overline{\beta}(S/I_\Delta):= \frac{1}{n} \beta(S/I_\Delta)$.\footnote{For a similar reason, \cite[Conjecture~B]{ein-erman-laz-random} also allows for a rescaling function.}

\begin{thm}\label{thm:normal}
Fix a constant $0<c<1$ and let $\Delta\sim \Delta(n,\frac{c}{n})$ be a random flag complex.  
 If $\{i_n\}$ is an integer sequence satisfying $i_n=n/2+o(n)$, and if $C:=\frac{1-c}{2}$, then
\[
\frac{\overline{\beta}_{i_n,i_n+1}(S/I_\Delta)}{C\binom{n}{i_n}} \longrightarrow 1 
\]
in probability.
\end{thm}

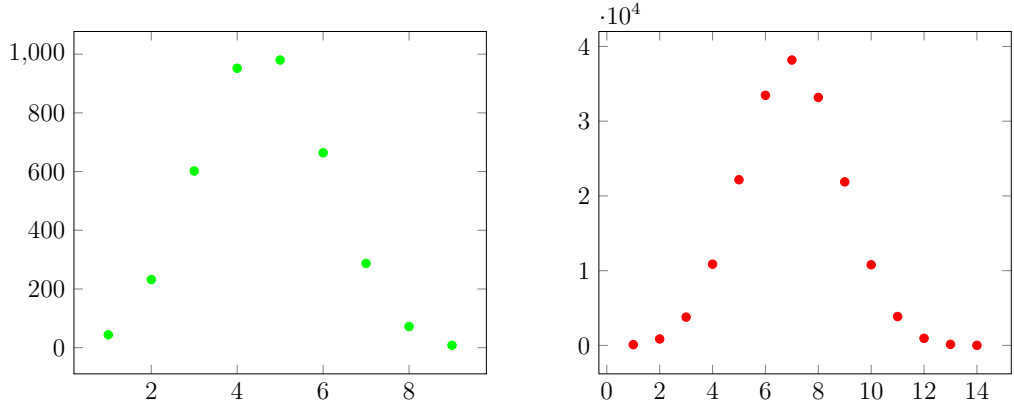
\begin{figure}
\[
 \begin{tikzpicture}[scale=0.8]
   \begin{axis}
   \addplot[only marks,color=green] coordinates {
     (1,44)
     (2,232)
     (3,602)
     (4,952)
     (5,980)
     (6,664)
     (7,287)
     (8,72)
     (9,8)
   };
   \end{axis}
 \end{tikzpicture}
 \ \qquad \
 \begin{tikzpicture}[scale=0.8]
   \begin{axis}
   \addplot[only marks,color=red] coordinates {
     (1,101)
     (2,858)
     (3,3783)
     (4,10868)
     (5,22165)
     (6,33462)
     (7,38181)
     (8,33176)
     (9,21879)
     (10,10790)
     (11,3861)
     (12,948)
     (13,143)
     (14,10)
   };
 \end{axis}
\end{tikzpicture}
\]
\caption{We plot the function $i\mapsto \beta_{i,i+1}(S/I_\Delta)$ for a random $\Delta \sim \Delta(10,\frac{1}{20})$ and $\Delta \sim \Delta(15,\frac{1}{30})$, respectively.  These appear consistent with the appearance of binomial coefficients, as in the heuristic of~\cite{ein-erman-laz-random} and in Theorem~\ref{thm:normal}.}
\end{figure}
Theorem~\ref{thm:normal}  is a local limit theorem, in the sense that it is a pointwise convergence rather than a global result about the whole distribution.  Moreover, the theorem is entirely focused on Betti numbers near the middle of the first row.  Yet, by a standard change of variables, this suffices to provide an example of the behavior predicted by~\cite[Conjecture~B]{ein-erman-laz-random}.

\begin{cor}\label{cor:normal2}
Fix a constant $0<c<1$ and let $\Delta\sim \Delta(n,\frac{c}{n})$ be a random flag complex.  
If $\{i_n\}$ is a sequence of integers converging to $\frac{n}{2} + \frac{a\sqrt{n}}{2}$, then
\[
\frac{\sqrt{2\pi}}{(1-c)2^n\sqrt{n}}\cdot\beta_{i_n,i_n+1}(S/I_\Delta)\longrightarrow e^{-a^2/2}
\]
in probability.
\end{cor}

The only previously known example of this kind comes from smooth curves \cite[Theorem~A]{ein-erman-laz-random}.  However, that example avoids the complexity of overlapping Betti numbers.   By contrast, for the family of ideals in Theorem~\ref{thm:normal}, the Betti numbers are not always clustered in a single row (see Remark~\ref{rmk:not single row}).  Thus, Theorem~\ref{thm:normal} produces the first known families of ideals which exhibit overlapping Betti numbers and behave like~\cite[Conjecture B]{ein-erman-laz-random}.
  
The following simple computation suggests why the Betti numbers of random flag complexes should behave like rescaled binomial coefficients.  For a subset $\alpha$ of the vertices, we write $\Delta|_{\alpha}$ for the restricted flag complex.  Hochster's formula~\cite[Theorem~5.5.1]{bruns-herzog} shows that $\beta_{i,i+1}(S/I_\Delta)$ is the sum over all $\alpha\in\binom{[n]}{i}$ of $\dim \HH_0(\Delta|_{\alpha})$.  By linearity of expectations, the expected value of $\beta_{i,i+1}(S/I_\Delta)$ is
\[
\bE\left[ \beta_{i,i+1}(S/I_\Delta)\right] = \sum_{\alpha \in \binom{[n]}{i}} \dim \HH_0(\Delta|_{\alpha}) = \binom{n}{i} \bE\left[\HH_0(\Delta')\right]
\]
where $\Delta' \sim \Delta(i,\frac{c}{n})$ is a random flag complex.  So it suffices to control how the expectation $ \bE\left[\HH_0(\Delta')\right]$ varies with $i$.  The main issue in proving Theorem~\ref{thm:normal} thus arises in showing convergence in probability, stemming from the fact that $\beta_{i,i+1}(S/I_\Delta)$ is a sum of dependent random variables.  

Not coincidentally, the choice $p=\frac{c}{n}$ (as in Theorem~\ref{thm:normal}) is a much-studied regime in the random graph literature.  See~\cite[\S11]{alon-spencer} or ~\cite[\S2.1]{frieze-book}, among other references.   We rely on some of those structural results about random graphs in this regime for our proofs of Proposition~\ref{prop:graphs} and Theorem~\ref{thm:normal}.   
\medskip

We also prove some results on the algebraic invariants of $S/I_\Delta.$  For instance, we prove the following threshold result for individual Betti numbers:
\begin{theorem}[Betti Number Thresholds]\label{thm:threshold}
Fix $i,v$ with $1\leq i$ and $i+1\leq v\leq 2i$ and let $s:=v-i-1$.  
Fix some constant $0<\epsilon\leq \frac{1}{2}$ and let $\Delta\sim\Delta(n,p)$.
\begin{enumerate}
  \item If $\frac{1}{n^{1/s}}\ll p\leq \epsilon$ then $\bP[\beta_{i,v}(S/I_\Delta)\neq 0]\to 1$.
   \item If $p\ll \frac{1}{n^{1/s}}$ then $\bP[\beta_{i,v}(S/I_\Delta)=0]\to 1$. 
  \end{enumerate}
\end{theorem}
\noindent We use this to bound the Castelnuovo-Mumford regularity of $S/I_\Delta$ in Corollary~\ref{cor:reg}.   Corollary~\ref{cor:CM and almost CM} also shows that while $S/I_\Delta$ is almost never Cohen-Macaulay, the depth and codimension of $S/I_\Delta$ converge as $n\to \infty$.  

This paper is organized as follows.  Section~\ref{sec:background} provides some essential definitions.  Section~\ref{sec:betti threshold} provides a threshold for the vanishing/nonvanishing of individual Betti numbers, the nonvanishing half of which relies on a variance bound proven Section~\ref{sec:variance bound}.  In Section~\ref{sec:ein-laz} we use the Betti number threshold to prove Theorem~\ref{thm:nonvanishing}.  In Section~\ref{sec:normal} we prove Theorem~\ref{thm:normal} and Corollary~\ref{cor:normal2}. Section~\ref{sec:pdim} contains estimates on the projective dimension of the ideal $I_\Delta$.

\section*{Acknowledgments}
We thank Juliette Bruce, Anton Dochterman, David Eisenbud, Gregory G.~Smith, and Zach Teitler for helpful conversations.  We also thank an anonymous referee, whose comments significantly improved the paper.  Many computations were done in Macaulay2~\cite{M2}.

\section{Background and Notation}\label{sec:background}
We work over an arbitrary field $k$.  We write $\bP[-]$ for the probability of an event and $\bE[-]$ for the expected value of a random variable.

A flag complex is a simplicial complex obtained from a graph by adjoining  a $k$-simplex to every $(k+1)$-clique in the graph.  We use $G\sim G(n,p)$ to denote an Erd\H{o}s-R\'enyi random graph on $n$ vertices, where each edge is attached with probability $p$, and we use $\Delta\sim\Delta(n,p)$ to denote the corresponding random flag complex.  If $H$ is a subset of the $n$ vertices, then we use $\Delta|_{H}$ for the induced flag complex.

The generators of $I_\Delta$ correspond to the maximal non-faces of $\Delta$~\cite[Chapter~5]{bruns-herzog}, and since $\Delta$ is flag this means that $I_\Delta$ is generated by quadrics.  Hochster's Formula~\cite[Theorem~5.5.1]{bruns-herzog}, which relates the Betti table of $S/I_\Delta$ to topological properties of $\Delta$, is our key tool for studying the syzygies of $S/I_\Delta$.

\begin{remark}
As discussed in the introduction, our goal is to use the $I_\Delta$ to model asymptotic syzygies.  The ideals of high degree Veroneses always admit a quadratic Gr\"obner basis~\cite{eisenbud-reeves-totaro}, and this is one reason why we chose to use random flag complexes.  By contrast, models in \cite{dpssw} often produce ideals with generators in different degrees, and those would thus provide better models for other families of examples.
\end{remark}

\begin{example}
Hochster's Formula implies that $\beta_{r+1,2r+2}(S/I_\Delta)$ is the number subcomplexes $\Delta|_H\subseteq \Delta$ where $H$ has $2r+2$ vertices and where ${\HH_r}(\Delta|_H)\ne 0$.
For instance $\beta_{1,2}(S/I_\Delta)$ is the pairs of disjoint vertices in $\Delta$, or equivalently it is the number of non-edges of the $\Delta$.  And $\beta_{2,4}(S/I_\Delta)$ is the number of squares in $\Delta$.  On the other hand, $\beta_{2,5}(S/I_\Delta)$ counts subcomplexes on five vertices with nonzero $\HH_1$.  There are several different types of examples, such as:
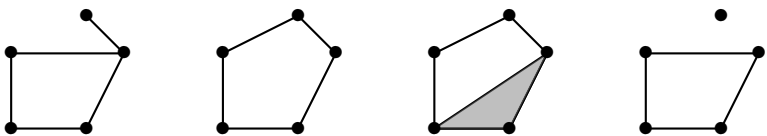
\[
\begin{tikzpicture}[scale=1.0]
\draw[-,thick](0,0)--(0,1)--(1.5,1)--(1,0)--cycle;
\draw[-,thick](1.5,1)--(1,1.5);
\draw (0,0) node {$\bullet$};
\draw (0,1) node {$\bullet$};
\draw (1.5,1) node {$\bullet$};
\draw (1,0) node {$\bullet$};
\draw (1,1.5) node {$\bullet$};
\end{tikzpicture}
\qquad 
\begin{tikzpicture}[scale=1.0]
\draw[-,thick](0,0)--(0,1)--(1,1.5)--(1.5,1)--(1,0)--cycle;
\draw (0,0) node {$\bullet$};
\draw (0,1) node {$\bullet$};
\draw (1.5,1) node {$\bullet$};
\draw (1,1.5) node {$\bullet$};
\draw (1,0) node {$\bullet$};
\end{tikzpicture}
\qquad
\begin{tikzpicture}[scale=1.0]
\draw[-,thick](0,0)--(0,1)--(1,1.5)--(1.5,1)--(1,0)--cycle;
\draw[-,thick](1,0)--(1.5,1)--(0,0)--cycle;
\fill[fill=gray,semitransparent](1,0)--(1.5,1)--(0,0)--cycle;
\draw (0,0) node {$\bullet$};
\draw (0,1) node {$\bullet$};
\draw (1.5,1) node {$\bullet$};
\draw (1,1.5) node {$\bullet$};
\draw (1,0) node {$\bullet$};
\end{tikzpicture}	
\qquad
\begin{tikzpicture}[scale=1.0]
\draw[-,thick](0,0)--(0,1)--(1.5,1)--(1,0)--cycle;
\draw (0,0) node {$\bullet$};
\draw (0,1) node {$\bullet$};
\draw (1.5,1) node {$\bullet$};
\draw (1,0) node {$\bullet$};
\draw (1,1.5) node {$\bullet$};
\end{tikzpicture}
\]
\end{example}

\begin{lemma}\label{lem:koszul}
If $\Delta$ is a flag complex, then 
$\beta_{i,j}(S/I_\Delta)=0$ for all $j>2i$.
\end{lemma}
\begin{proof}
Since $\Delta$ is flag, $I_\Delta$ is a monomial ideal generated by quadrics.  The Taylor resolution of $S/I_\Delta$ thus involves monomials of degree $0,1,$ or $2$~\cite[Construction~26.5]{peeva}.
\end{proof}

The boundary complex of the $(r+1)$-dimensional octahedron plays a key role in our results (for instance, see Remark~\ref{rmk:fewest}), and we denote this flag complex by $\oct_r$.  We note that $\oct_r$ is also the $r$-fold suspension of $2$ points.  See Figure~\ref{fig:oct}. Since a pair of points is disconnected, we have $\HH_0(\oct_0)\cong \ZZ$, and since taking suspensions shifts reduced homology groups up by one degree, we have that $\HH_r(\oct_r)\cong \ZZ$.  We now observe that any flag complex with nonzero $r$th homology will have at least as many vertices and edges as $\oct_r$. 
\begin{figure}
\begin{tikzpicture}[scale=1.0]
\draw (0,.5) node {$\bullet$};
\draw (1,.5) node {$\bullet$};
\draw (.5,-.75) node {$\oct_0$};
\end{tikzpicture}
\qquad
\begin{tikzpicture}[scale=1.0]
\draw[-,thick](0,0)--(0,1)--(1,1)--(1,0)--cycle;
\draw (0,0) node {$\bullet$};
\draw (0,1) node {$\bullet$};
\draw (1,0) node {$\bullet$};
\draw (1,1) node {$\bullet$};
\draw (.5,-.75) node {$\oct_1$};
\end{tikzpicture}
\qquad
\begin{tikzpicture}[xscale=1.4,yscale=1.2]
\coordinate (A1) at (0,.5);
\coordinate (A2) at (0.8,.7);
\coordinate (A3) at (1,.5);
\coordinate (A4) at (0.2,.3);
\coordinate (B1) at (0.5,1.4);
\coordinate (B2) at (0.5,-0.4);
\draw (A1) node {$\bullet$};
\draw (A2) node {$\bullet$};
\draw (A3) node {$\bullet$};
\draw (A4) node {$\bullet$};
\draw (B1) node {$\bullet$};
\draw (B2) node {$\bullet$};
\draw[dashed] (A1) -- (A2) -- (A3);
\draw[-,thick] (A1) -- (A4) -- (A3);
\draw[dashed] (B1) -- (A2) -- (B2);
\draw[-,thick] (B1) -- (A4) -- (B2);
\draw[-,thick] (B1) -- (A1) -- (B2) -- (A3) --cycle;
\draw[-,thick] (.5,-.75) node {$\oct_2$};
\fill[fill=gray,semitransparent](B1)--(A1)--(B2)--(A3)--cycle;
\end{tikzpicture}
\caption{Among flag complexes with nonzero $r$th homology, the boundary complex of the $(r+1)$-dimensional octahedron, which we denote $\oct_r$, has the fewest edges.}
\label{fig:oct}
\end{figure}
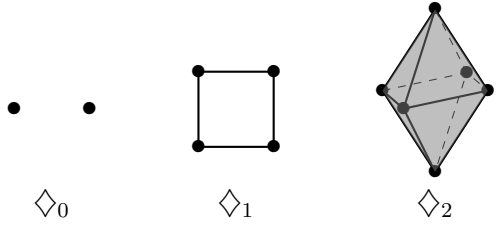

\begin{lemma}\label{lem:HHr}
Let $\Delta$ be a flag complex with $\HH_r(\Delta)\neq 0$.  
\begin{enumerate}
	\item\label{lem:HHr:vertices}
 Then $\Delta$ has at least $2r+2$ vertices.
 	\item\label{lem:HHr:vertex degree}  If $v\in \Delta$ is a vertex such that $\HH_r(\Delta_{\Delta-v})=0$, then $\deg(v)\geq 2r$.
	\item\label{lem:HHr:edges}  $\Delta$ has at least $2r(r+1)$ edges.
\end{enumerate}
\end{lemma}
\begin{proof}
This result is folklore.  Part~\eqref{lem:HHr:vertices} is proven in~\cite[Lemma 3.6]{cjkw}.  Parts~\eqref{lem:HHr:vertex degree} and \eqref{lem:HHr:edges} follow easily by standard topological arguments.
\end{proof}
%
\begin{remark}
The complex $\oct_r$ shows that the bounds in Lemma~\ref{lem:HHr} are sharp.
\end{remark}


\section{Variance Bound}\label{sec:variance bound}
In this section we prove a variance bound that is used in our convergence results. The proof is similar to those in~\cite[Theorem~1]{bollobas-erdos} and ~\cite[Lemma~2.2]{kahle} and elsewhere.  

\begin{remark}\label{rmk:fewest}
We are particularly interested in the appearance of subcomplexes of the form $\oct_s$, as by Lemma~\ref{lem:HHr} these are the flag complexes with the fewest edges and nonzero $s$th homology.  Since in our models $p$ goes to $0$ as $n\to \infty$, subcomplexes with fewer edges are more likely to appear, and so we expect these $\oct_s$ to control $(s+1)$th row of $\beta(S/I_\Delta)$.
\end{remark}
\begin{remark}\label{rmk:indexing}
In $\oct_s$, every vertex has a unique antipodal vertex, and thus as a subgraph of $\Delta,$ $\oct_s$ is determined by $s+1$ pairs of vertices, all distinct.  In particular, given a set of vertices $V\in \binom{[n]}{2(s+1)}$ there are multiple ways that $\Delta|_V$ could be an $\oct_s$-subcomplex; to simplify the computations in this section, it will be useful to parametrize each potential $\oct_s$ separately, even those that involve the same vertices.  We define $\Lambda_s$ as vertex sets $V\in \binom{[n]}{2(s+1)}$ of size $2(s+1)$ together with an unordered decomposition $V=P_0\cup \dots \cup P_s$, where each $P_i$ is an unordered pair of vertices.  With this definition, there is then a bijection between elements of $\Lambda_s$ and potential subcomplexes $\oct_s\subseteq \Delta$.  Thus, given any $H\in \Lambda_s$, the probability that $\Delta|_H$ is $\oct_s$ is given precisely by probability that $\Delta|_H$ has exactly the specified edges, which is $p^{2s(s+1)}(1-p)^{\binom{2(s+1)}{2} - 2s(s+1)}$.
\end{remark}
\begin{defn}\label{defn:X}
Let $\rX_s=\rX_s(n,p)$ denote the random variable for the number of copies of $\oct_s$ appearing as a subgraph of a random graph $G\sim G(n,p)$.  Given $H\in \Lambda_s$ we then define $\rX_H$ as the indicator random variable for whether the subgraph on $H$ has the form $\oct_s$.  
\end{defn}
Thus we have $\rX_s = \sum_{H\in \Lambda_s} \rX_H$. We will now use this to bound the variance $\Var[\rX_s]$.
\begin{lemma}[Variance Bound]\label{lem:variance bound}
If $np^{\left(s+\frac{1}{2}\right)}\rightarrow \infty$ and $p\leq(1-p)$, then 
  $\frac{\Var[\rX_s]}{\bE[\rX_s]^2}\rightarrow 0$.
\end{lemma}
\begin{proof}
We start by computing
  \begin{align*}
    \bE[\rX_s^2]
    &=\sum_{H,J\in \Lambda_s} \bE[\rX_H\rX_J]\\
    &=\sum_{H,J\in \Lambda_s} \bP[\rX_J=1|\rX_H=1]\bP[\rX_H=1]\\
    &=\sum_{H\in \Lambda_s} \bP[\rX_H=1]\sum_{J\in \Lambda_s} \bP[\rX_J=1|\rX_H=1].\\
\intertext{Since $\sum_{J\in \Lambda_s} \bP[\rX_J=1|\rX_H=1]$ is independent of the choice of $H$, we may fix an $H'$ to decouple the factors, yielding}
&=\left(\sum_{H\in \Lambda_s} \bP[\rX_H=1]\right)\sum_{J\in \Lambda_s} \bP[\rX_J=1|\rX_{H'}=1]\\
    &=\bE[\rX_s]\bE[\rX_s|\rX_{H'}=1].
  \end{align*}
 Since $\Var[\rX_s] = \bE[\rX_s^2] - \bE[\rX_s]^2$, the above computation allows us to compute:
  \begin{align*}
    \Var(\rX_s)/\bE[\rX_s]^2 &= \frac{\bE\big[\rX_s|\rX_H=1\big]-\bE\big[\rX_s\big]}{\bE[\rX_s]}\\
    & = \frac{\sum_{m=0}^{2s+2}\sum_{\left|J\cap H\right|=m}\bP\big[\rX_J=1|\rX_H=1\big]-\bP\big[\rX_J=1\big]}{\bE[\rX_s]}.\\
    \intertext{If $J$ and $H$ are disjoint or intersect in only a single vertex, then $\bP\big[\rX_J=1|\rX_H=1\big]=\bP\big[\rX_J=1\big]$.  We can thus ignore the terms with $m=0$ or $m=1$ in this sum:}
    & = \frac{\sum_{m=2}^{2s+2}\sum_{\left|J\cap H\right|=m}\bP\big[\rX_J=1|\rX_H=1\big]-\bP\big[\rX_J=1\big]}{\bE[\rX_s]}.\\
 \intertext{By Lemma~\ref{lem:conditional}, we obtain the bound}
    & \leq \frac{\sum_{m=2}^{2s+2}\sum_{\left|J\cap H\right|=m}p^{-m(m-1)/2}\bP\big[\rX_J=1\big]-\bP\big[\rX_J=1\big]}{\bE[\rX_s]}.\\
        \intertext{Since the probability $\bP[\rX_J=1]$ does not depend on $J$, we can use the bound from Lemma~\ref{lem:quotient} to pull $\bP[\rX_J=1]/\bE[\rX_s]$ outside, and simplify the expression, where $C$ is a constant:}
    & \leq Cn^{-2(s+1)}\sum_{m=2}^{2s+2}\sum_{\left|J\cap H\right|=m}p^{-m(m-1)/2}-1.\\
    \intertext{Up to a constant, for a fixed $H$ there are $n^{2(s+1)-m}$ choices of $J$ where $|J\cap H|=m$.  Absorbing those constants into our $C$ we get:}
    & \leq Cn^{-2(s+1)}\sum_{m=2}^{2s+2}n^{2(s+1)-m}(p^{-m(m-1)/2}-1)\\
    & = C\sum_{m=2}^{2s+2}n^{-m}(p^{-m(m-1)/2}-1)\\
    & \leq C\sum_{m=2}^{2s+2}(np^{(m-1)/2})^{-m}.
  \end{align*}
Since $0<(m-1)/2\leq s+\frac{1}{2}$ we have $np^{(m-1)/2}\rightarrow \infty$ by hypothesis. It follows that all of the finitely many terms in the sum go to 0, and thus $\Var(\rX_s)/\bE[\rX_s]^2\rightarrow 0$. 
\end{proof}

 \begin{lemma}\label{lem:conditional}
    Given $J,H\in \Lambda_s$ such that $\left|J\cap H\right| = m$
    \[\bP\big[\rX_J=1|\rX_H=1\big]\leq p^{-m(m-1)/2}\bP\big[\rX_J=1\big].\]
  \end{lemma}
  \begin{proof}
If $\rX_H=1$ then the edges in $J\cap H$ are completely determined. If those edges do not match the required edges for $J$, then $\bP[\rX_J=1|\rX_H=1]=0$. If they do match the required edges, then since the probability of any edge existing or not existing is $p$ or $1-p$, and since $p\leq 1-p$, we get that $\bP[\rX_J=1|\rX_H=1]\leq p^{-m(m-1)/2}\bP[\rX_J=1]$
  \end{proof}
\begin{lemma}\label{lem:quotient}
For any fixed $H\in \Lambda_s$, we have $\bP\big[\rX_H=1\big]/\bE[\rX_s]\leq Cn^{-2(s+1)}$ for some constant $C$.
\end{lemma}
\begin{proof}
Since $\rX_s=\sum_H \rX_H$ we have $\bE[\rX_s] = \sum_H \bP[\rX_H=1]$.  But since $\bP[\rX_H=1]$ does not depend on $H$, this amounts to counting the number of possible choices of $H$, which is the cardinality of $\Lambda_s$.  Each element of $\Lambda_s$ corresponds to $s+1$ pairs of verticies in $\Delta$, of which there $\frac{1}{(s+1)!} \binom{n}{2,2,2,\dots, n-2(s+1)}$ choices.  It follows that, for an appropriate constant $C$, we have $\bP[\rX_H=1]/\bE[\rX_s]\leq Cn^{-2(s+1)}$.
\end{proof}

\section{Betti Number Thresholds}\label{sec:betti threshold}
In this section, we determine thresholds of nonvanishing for individual Betti numbers.  
Lemma~\ref{lem:koszul} shows that $\beta_{i,v}(S/I_\Delta)=0$ whenever $v\leq i$ or $v\geq 2i$,  and Theorem~\ref{thm:threshold} computes thresholds in the remaining cases.  To prove that theorem, we first bound the expected values of the Betti numbers.
For $\Delta\sim\Delta(n,p)$ we define $\rB_{i,v}$ where $\rB_{i,v}(\Delta):=\beta_{i,v}(S/I_\Delta)$. 
By convention, when $s=0$ we interpret $\frac{1}{n^{1/s}}\ll p$ as a trivial bound.

\begin{lemma}\label{lem:E X}
Fix any constant $0<\epsilon<1$.  Let $\frac{1}{n^{1/s}}\ll p \leq \epsilon$ and $\Delta\sim \Delta(n,p)$.  We have $\bE[\rB_{s+1,2s+2}]\rightarrow \infty$ as $n\to \infty$.
\end{lemma}
\begin{proof}
By Hochster's formula~\cite[Theorem~5.5.1]{bruns-herzog}, since $\HH_s(\oct_s)\ne 0$, we have $\bE[\rB_{s+1,2s+2}] \geq \sum_{H} \bE[\rX_H]$, where as in Definition~\ref{defn:X}, $H$ is a set of $s+1$ pairs of vertices, all distinct.  Since any $\oct_s$ involves $s(2s+2)$ edges and $s+1$ non-edges, we have
  \[
 \bE[\rX_H]=\bP[\rX_H=1] = p^{s(2s+2)}(1-p)^{s+1}.
 \]
As in the proof of Lemma~\ref{lem:quotient}, the number of choices for $H$ is at least $Cn^{2s+2}$ for some positive constant $C$, and thus
\[
    \bE[\rB_{s+1,2s+2}]=\sum_H \bE[\rX_H]\geq 
    Cn^{2s+2} p^{s(2s+2)}(1-p)^{s+1}\geq C' (np^{s})^{2s+2}.
\]
 where $C'=C(1-\epsilon)^{s+1}$.
Since $np^{s}\to \infty$ it follows that $\bE[\rB_{s+1,2s+2}]\to \infty$.
\end{proof}

To prove the other threshold, we introduce new random variables.
\begin{defn}
Let $\rY^s_v=\rY^s_v(n,p)$ be the number of subgraphs with $m\leq v$ vertices and at least $ms$ edges.  If $K$ is a subset of $m$ vertices, we let $\rY^s_K$ be the indicator random variable for whether the subgraph on $K$ has at least $ms$ edges.
\end{defn}

\begin{lemma}\label{lem:exp Y}
If $p\ll 1/n^{1/s}$
then $\bE[\rB_{i,v}]\to 0$.
\end{lemma}
\begin{proof}
Lemma~\ref{lem:HHr} shows that if $K$ is a minimal subset of vertices of $\Delta$ such that $\HH_s(\Delta|_K)\ne 0$, then each vertex in $\Delta|_K$ has degree $\geq 2s$.  In particular, if $\beta_{i,v}(S/I_\Delta)\ne 0$, then there must exist some subgraph $K$ of size at most $v$ (and with at least $2s+2$ vertices) where every vertex has degree $\geq 2s$.  It thus suffices to prove that $\bE[\rY_v^s]\to 0$.

We have $\rY_v^s = \sum_{K, |K|\leq v} \rY_K^s$.  For a fixed $K$ with $|K|=m$, we want to compute the probability that $\Delta|_K$ has at least $ms$ edges.  We use $M:=\binom{m}{2}$ to denote the maximal number of possible edges.  We thus have
  \[\bP[Y^s_K=1] = \sum_{e=ms}^{M}\binom{M}{e}p^{e}(1-p)^{M-e}.\]
We then compute:
  \begin{align*}
  \bE[Y^s_v] &= \sum_{m=2s+2}^v \sum_{K, |K|=m}  \bP[Y^s_K = 1]\\
  &= \sum_{m=2s+2}^{v}\binom{n}{m}\sum_{e=ms}^{M}\binom{M}{e}p^{e}(1-p)^{M-e}\\
  &\leq \sum_{m=2s+2}^{v}\binom{n}{m}\sum_{e=ms}^{M}\binom{M}{e}p^{e}\\
  &\leq \sum_{m=2s+2}^{v}\binom{n}{m}p^{ms}\sum_{e=ms}^{M}\binom{M}{e}p^{e-ms}.\\
  \intertext{However, we can bound $\sum_{e=ms}^{M}\binom{M}{e}p^{e-ms}$ by a constant $C_{s,m}$ depending only on $s$ and $m$, and we can bound $\binom{n}{m}$ by $n^m$.  This yields: }
  &\leq \sum_{m=2s}^{v}n^mp^{ms}C_{s,m}=\sum_{m=2s}^{v}\left(np^{s}\right)^mC_{s,m}.\\
  \end{align*}
Finally, since $np^s\rightarrow 0$ by assumption, we conclude that $\bE[Y^s_v]\rightarrow 0$.
\end{proof}
\begin{proof}[Proof of Theorem~\ref{thm:threshold}]
For statement (1), we first consider the case where $v=2i=2s+2$.  Lemma~\ref{lem:E X} implies that  $\bE[\rB_{s+1,2s+2}]\to \infty$.   Thus to prove that $\bP[\rB_{s+1,2s+2}\ne 0]\to 1$, we may bound the variance of $\rB_{s+1,2s+2}$.  This is done in Lemma~\ref{lem:variance bound} since $\rB_{s+1,2s+2}=\rX_s$. There we show that $\frac{\Var[\rB_{s+1,2s+2}]}{\bE[\rB_{s+1,2s+2}]^2}\rightarrow 0$. Thus we can apply Chebyshev's Inequality to say the following.
\begin{align*}
\bP\left[\rB_{s+1,2s+2}= 0\right] &\leq \bP\left[\left| \bE\left[\rB_{s+1,2s+2}\right]- \rB_{s+1,2s+2}\right|\geq \bE\left[\rB_{s+1,2s+2}\right]\right]\\
 &\leq  \frac{\Var[B_{s+1,2s+2}]}{(\bE[\rB_{s+1,2s+2}]-1)^2}\rightarrow 0
\end{align*}


We now let $v<2i$.  The case $v=2s+2$ implies the existence of some $\oct_s\subseteq \Delta$ with probability $1-o(1)$.  Fix some vertex  $u\in \oct_s$.  Let $J$ be the set of vertices $w\in \Delta$ which don't lie in $\oct_s$ and which are not connected with $u$.  Since the complement of $\oct_s$ consists of $n-(2s+2)$ vertices, the expected number of vertices in $J$ is $(n-(2s+2))(1-p)=n-o(n)$.  Moreover, since those conditions are independent, the Weak Law of Large Numbers implies that this happens with high probability.  Let $J'\subseteq J$ be any subset of cardinality $v-(2s+2)$.  Since the only edges in $\oct_s\cup J'$ through the vertex $u$ are the ones from $\oct_s$, it follows $\HH_s(\oct_s\cup J')$ is still nonzero.  Hence $\rB_{i,v}\ne 0$ with high probability as desired.

For statement (2), we must show that $\rB_{i,v}$ converges to $0$ in probability. Hochster's formula~\cite[Theorem~5.5.1]{bruns-herzog} implies that $\beta_{i,v}(S/I_\Delta)$ is nonzero if and only there is some subset $K\subseteq \Delta$ with $|K|=v$ and where $\HH_{v-i-1}(\Delta|_{K})\ne 0$.   By Lemma~\ref{lem:HHr} it suffices to show that $\bP[Y^s_v=0] \rightarrow 1$ for $s=v-i-1$. But by Lemma~\ref{lem:exp Y}, we know $\bE[Y^s_v]\rightarrow 0$, and since $Y^s_v\geq 0$ and $Y^s_v$ takes integer values, this implies that $\bP[Y^s_v=0]\rightarrow 1$.
\end{proof}

\section{Ein-Lazarsfeld Asymptotic Nonvanishing of Syzygies}\label{sec:ein-laz}
Whereas Theorem~\ref{thm:threshold} provides the nonvanishing thresholds for individual Betti numbers, Question~\ref{q:nonvanishing density} asks about the simultaneous nonvanishing of more and more Betti numbers as $n\to \infty$.  However, as we now illustrate, the proof of Theorem~\ref{thm:threshold} is sufficiently strong to obtain simultaneous nonvanishing of the various Betti numbers.

\begin{proof}[Proof of Theorem~\ref{thm:nonvanishing}]
For each $n$, we partition the vertices into $r+1$ sets $S_0, S_1, \dots, S_r$ each of size approximately $\frac{n}{r+1}$.  Since $\Delta|_{S_s}$ is a random flag complex for any $0\leq s \leq r$, the proof of Theorem~\ref{thm:threshold} implies the existence of some $\oct_{s}$ in $\Delta|_{S_s}$ with probability $1-o(1)$.  Moreover, since $r$ is fixed, we can assume that that these all occur simultaneously.  By construction, the vertices involved in $\oct_0, \oct_1, \dots, \oct_r$ are all disjoint.

Fix some $0<\epsilon<1$.  For each $0\leq s\leq r$, fix some vertex  $v\in \oct_s$.  Since the complement of $\cup_{s=0}^r \oct_s$ consists of $n-O(1)$ vertices, the expected number of vertices $w\notin \cup_{s=0}^r \oct_s$ that are not connected with vertex $v$ is $(n-O(1))(1-p) \geq n-n^{1-\epsilon}$, at least for $n$ sufficiently large.  Since those conditions are independent, the Weak Law of Large Numbers implies that this happens with high probability.  Call that set $J$ and $J'\subseteq J$ be any subset.  Since the only edges in $\oct_s\cup J'$ through the vertex $v$ are the ones from $\oct_s$, it follows $\HH_s(\Delta|_{\oct_s\cup J'})$ is still nonzero.  Since $|\oct_s\cup J'|$ ranges from $2s+2$ to $n-n^{1-\epsilon}+2s+2$, it follows that $\beta_{i+1,i+s+2}(S/I_\Delta)\ne 0$ for all $s\leq i \leq n-n^{1-\epsilon}+s$ with high probability.  In particular, with high probability we have
\[
\lim_{n\to \infty} \rho_{s+1}(S/I_\Delta) \geq \lim_{n\to \infty} \frac{n-n^{1-\epsilon} +1}{n} = 1.
\]
Moreover, since the $\oct_s$ involve disjoint vertices, these nonvanishing conditions are independent in $s$, and we thus obtain the desired convergence of $\rho_{s+1}$ for all $s$ simultaneously.
\end{proof}

The proof of Theorem~\ref{thm:nonvanishing} shows that if we cross the threshold for the appearance of subcomplexes of the form $\oct_s$, then we get nonvanishing across nearly the entire $(s+1)$th row of the Betti table.  The appearance of $\oct_s$ subcomplexes thus accounts for why $\rho_{s+1}(S/I_\Delta)$ goes to $1$.

\begin{example}\label{ex:nonvanishing}
Here is the Betti table of $S/I_\Delta$ for a randomly chosen $\Delta\sim\Delta(18,\frac{1}{18^{0.6}})$, as computed in Macaulay2~\cite{M2}.
{\tiny
\[
\begin{array}{*{18}c}
0&1&2&3&4&5&6&7&8&9&10&11&12&13&14&15&16&17\\
1&\text{.}&\text{.}&\text{.}&\text{.}&\text{.}&\text{.}&\text{.}&\text{.}&\text{.}&\text{.}&\text{.}&\text{.}&\text{.}&\text{.}&\text{.}&\text{.}&\text{.}\\
\text{.}&126&1203&5986&19491&45278&78385&103667&106356&85548&54408&27541&11118&3550
&873&156&18&1\\
\text{.}&\text{.}&1&24&233&1282&4568&11261&19911&25743&24538&17229&8815&3204&786&117&8&\text{.}\\
\end{array}
\]
}
As predicted by Theorem~\ref{thm:nonvanishing}, the entries in rows $1$ and $2$ are almost all nonzero.
\end{example}

Though we do not compute a precise threshold for the Castelnuovo-Mumford regularity of $S/I_{\Delta}$, we do obtain a linear bound.
\begin{cor}\label{cor:reg}
If $\frac{1}{n^{1/r}}\ll p \ll \frac{1}{n^{2/(2r+1)}}$, then with high probability $r+1\leq \reg(S/I_{\Delta})\leq 2r$.
\end{cor}
\begin{proof}
Since $\frac{1}{n^{1/r}}\ll p$ we have that $\beta_{r+1,2r+2}(S/I_{\Delta})\ne 0$  and thus $\reg(S/I_\Delta)\geq r$, with high probability.  For the other direction, we let $s=2r+1$ so that $p\ll \frac{1}{n^{2/s}}$.  A simple computation shows that the expected number of $(s+1)$-cliques in $\Delta$ is
\[
\binom{n}{s+1}p^{\binom{s+1}{2}}\leq n^{s+1}\left(p^{s/2}\right)^{s+1}\ll n^{s+1}(n^{-1})^{s+1}=1.
\]
Since the expected number of $(s+1)$-cliques goes to zero, it follows that with high probablility $\Delta$ has no subcomplex with $(s+1)$th homology and thus $\reg(S/I_\Delta)<s=2r+1$.
\end{proof}

\begin{question}
Does $\reg(S/I_\Delta)$ converge in probability (with appropriate conditions on $p$)?  More precisely, if $\frac{1}{n^{1/r}}\ll p \ll \frac{1}{n^{1/(r+1)}}$ does $\reg(S/I_\Delta)$ converge to $r+1$ in probability?
\end{question}


\section{Normal Distribution of Quadratic Strand}\label{sec:normal}
In this section, we prove Theorem~\ref{thm:normal} and Corollary~\ref{cor:normal2}.
\begin{remark}\label{rmk:not single row}
For $\Delta$ as in Theorem~\ref{thm:normal}, the second row of the Betti table of $S/I_\Delta$ is interesting as well, because $p=c/n$ is a boundary case for the nonvanishing in Theorem~\ref{thm:nonvanishing}.  In~\cite[Theorem~5b]{erdos-renyi}, they prove that the $1$-skeleton of $\Delta$ will contain a cycle with probability $1-\sqrt{1-c}e^{(c/2)+(c^2/4)}$.  Among graphs containing at least one cycle, an argument similar to the proof of Theorem~\ref{thm:nonvanishing} yields $n-n^{1-\epsilon}$ nonzero entries in the second row of the Betti table of $S/I_\Delta$, and thus in this case, $S/I_\Delta$ will have overlapping Betti numbers throughout two rows, similar to the case of a smooth surface in Theorem~\ref{thm:EL}.
\end{remark}

Given a graph $G$, we define
\[
\mathbf{H}_0(G,k)= \sum_{\alpha \in \binom{[n]}{k}} \HH_0(G|_{\alpha})
\]
as the sum of $\HH_0(G|_{\alpha})$, where $\alpha\in \binom{[n]}{k}$ is a subset of the vertices of size $k$ and where $G|_\alpha$ is the induced subgraph.  Hochster's formula~\cite[Theorem~5.5.1]{bruns-herzog} implies that if $I_\Delta$ is a Stanley-Reisner ideal, then the Betti number $\beta_{k,k+1}(S/I_\Delta)$ equals $\mathbf{H}_0(G,k)$, where $G$ is the one-skeleton of the simplicial complex $\Delta$.  We can thus reduce Theorem~\ref{thm:normal} to the following computation about graphs.
\begin{prop}\label{prop:graphs}
Let $G\sim G(n,\frac{c}{n})$ be a random graph, with $0<c<1$.  If $\{i_n\}$ is an integer sequence satisfying $i_n=n/2+o(n)$, and if $C:=\frac{1-c}{2}$, then
  \[\frac{\mathbf{H}_0(G,i_n)}{Cn\binom{n}{i_n}}\rightarrow 1\]
  in probability.
\end{prop}

\begin{proof}

If we remove graphs from the distribution $G\in G(n,p)$ which arise with probability $o(1)$, then this will not affect facts about convergence in probability.  For instance, with probability $1-o(1)$ a random $G\sim G(n,\frac{c}{n})$ with $c<1$ will be the disjoint union of trees and components with a single cycle~\cite[p.~31]{frieze-book}.  Thus, we may restriction attention to graphs $G$ which are the disjoint union of trees and components with a single cycle.  Moreover, since the expected number of cycles is constant when $c<1$, we conclude that with probability $1-o(1)$, $\Delta$ has at most $n^{1-\epsilon}$ cycles for any fixed $0<\epsilon<1$.  We thus further restrict attention to the case where $\Delta$ is the disjoint union of trees and at most $n^{1-\epsilon}$ components each with a single cycle.  We denote this restricted distribution of graphs by $\widetilde{G}(n,\frac{c}{n})$ and we henceforth choose $G\sim \widetilde{G}(n,\frac{c}{n})$.

To prove the main result, we introduce several auxiliary random variables.  For a graph $G$, we now set $\rE(G)$ to be the number of edges in $G$ and we define $\rC(G)$ to be the number of cycles in $G$.  Finally, for a pair of vertices $e\in \binom{[n]}{2}$, we define $\rZ_{e}$ to be the indicator random variable of whether that pair of vertices is an edge in $G$.  

With this notation, and using our assumption that $G$ is a disjoint union of trees and components containing a single cycle, we have
\begin{align*}
\mathbf{H}_0(G,i_n) &= \sum_{\alpha \in \binom{[n]}{i_n}} i_n - \rE(G|_{\alpha}) + \rC(G|_{\alpha}).\\
\intertext{Ignoring the cycles, we get}
&\geq \sum_{\alpha \in \binom{[n]}{i_n}} i_n  - \rE(G|_{\alpha})\\
&=\binom{n}{i_n}i_n -  \sum_{\alpha \in \binom{[n]}{i_n}} \rE(G|_{\alpha}).\\
\intertext{We may rewrite the righthand sum in terms of individual edges to obtain}
&=\binom{n}{i_n}i_n -  \sum_{e\in \binom{[n]}{2}}\binom{n}{i_n-2}\rZ_e.\\
\intertext{But $\rE(G)$ is the sum of the $\rZ_e$, and thus we have:}
&=\binom{n}{i_n}i_n -  \binom{n}{i_n-2}\rE(G).\\
\end{align*}
By a similar argument, but where we do not ignore $\rC(G|_{\alpha})$, we can use the fact that $G$ has at most $n^{1-\epsilon}$ cycles to obtain an upper bound $\mathbf{H}_0(G,i_n)\leq \binom{n}{i_n}i_n - \binom{n}{i_n-2}(\rE(G)-n^{1-\epsilon})$.  
\begin{equation}\label{eqn:bounds}
\binom{n}{i_n}i_n - \binom{n}{i_n-2}\rE(G)
\leq \mathbf{H}_0(G,i_n)\leq \binom{n}{i_n}i_n - \binom{n}{i_n-2}\left(\rE(G)-n^{1-\epsilon}\right).
\end{equation}
We have $\binom{n}{i_n-2}=\binom{n}{i_n} \frac{i_n(i_n-1)}{(n-i_n+2)(n-i_n+1)}$ and since $i_n=n/2+o(n)$ this yields that $\binom{n}{i_n-2}=\binom{n}{i_n} (1+o(1))$.  Applying this to Equation~\eqref{eqn:bounds} yields:
\[
\binom{n}{i_n}\bigg( i_n - (1+o(1))\rE(G)\bigg)
\leq \mathbf{H}_0(G,i_n)\leq \binom{n}{i_n}\bigg(i_n -(1+o(1))\left(\rE(G)-n^{1-\epsilon}\right)\bigg).
\]
Recall that $C=\frac{1-c}{2}$.  We now divide through by $\frac{1}{Cn\binom{n}{i_n}}$.  By rewriting $i_n= n/2+o(n)$ and absorbing the $n^{1-\epsilon}$ term into the $o(n)$, the lefthand and righthand bounds have the same form, and we obtain
\[
\frac{\mathbf{H}_0(G,i_n)}{Cn\binom{n}{i_n}}=\frac{(n/2) - \rE(G) + o(n) + o(1)\rE(G)}{Cn}
\]
Since $\rE(G)$ is a sum of independent random variables, one for each potential edge, this now essentially reduces to a weak law of large numbers argument.  In particular, we have that the variance of $\rE(G)$ is $\binom{n}{2}p(1-p)$ and the mean is $\binom{n}{2}p=\frac{c(n-1)}{2}$. We apply Chebyshev's Inequality to the random variable $\rE(G)/n$:
\[ \bP\Bigg[\left|\frac{c(n-1)}{2n}-\frac{\rE(G)}{n}\right|\geq \epsilon\Bigg]\leq \frac{\Var(\rE(G)/n)}{\epsilon^2} = \frac{\binom{n}{2}p(1-p)}{n^2\epsilon^2}.\]
Since $p=\frac{c}{n}$ and $1-p<1$ this simplifies to $\frac{n(n-1)\frac{c}{n}}{2n^2\epsilon^2}$ which in turn reduces to $\frac{c(n-1)}{2n^2\epsilon^2}$.  For fixed $\epsilon$ we have $\lim_{n\rightarrow \infty} \frac{c(n-1)}{2n^2 \epsilon^2}=0$. Since $\lim_{n\rightarrow \infty}\frac{c(n-1)}{2n} = \frac{c}{2}$, we conclude that $\rE(G)/n$ converges to $\frac{c}{2}$ in probability. This implies that
$
\frac{\mathbf{H}_0(G,i_n)}{Cn\binom{n}{i_n}}\to 1
$
in probability.
\end{proof}

\begin{proof}[Proof of Theorem~\ref{thm:normal}]
Let $G$ be the $1$-skeleton of $\Delta$.  By Hochster's formula~\cite[Theorem~5.5.1]{bruns-herzog}, $\beta_{i_n,i_n+1}(S/I_\Delta)= \mathbf{H}_0(G,i_n)$.  The statement is now an immediate corollary of Proposition~\ref{prop:graphs}.
\end{proof}
\begin{proof}[Proof of Corollary~\ref{cor:normal2}]
Let $C=\frac{1-c}{2}$.  Using Theorem~\ref{thm:normal}(2) and the normal approximation of the binomial distribution, e.g.~\cite[(8.3), p. 762]{boas}, we obtain that
\[
\beta_{i_n,i_n+1}(S/I_\Delta) \sim Cn\binom{n}{i_n} \sim Cn\frac{2^{n+1}}{\sqrt{2\pi n}}e^{-a^2/2}.
\]
Therefore we have
\[
\frac{\sqrt{2\pi n}}{Cn2^{n+1}}\beta_{i_n,i_n+1}(S/I_\Delta)=\frac{\sqrt{2\pi}}{(1-c)2^n\sqrt{n}}\beta_{i_n,i_n+1}(S/I_\Delta)\sim e^{-a^2/2}.
\]
Since the righthand side is a constant, we have convergence in probability.
\end{proof}
%
\begin{conj}
In cases where Theorem~\ref{thm:nonvanishing} yields nonvanishing Betti numbers in row $k$, we conjecture that the $k$th row of the Betti table will be normally distributed, in a manner similar to Corollary~\ref{cor:normal2}.
\end{conj}

\section{Projective Dimension Estimates}\label{sec:pdim}
We conclude with a corollary about Cohen-Macaulayness.  For many values of $p$, we show that $S/I_\Delta$ will essentially never be Cohen-Macaulay.  However, while the projective dimension almost never equals the codimension of $S/I_\Delta$, with high probability the ratio of these quantities converges to $1$ as $n\to \infty$.

\begin{cor}\label{cor:CM and almost CM}
For any $k\geq 1$, and any $p$ satisfying $\frac{1}{n^{2/3}}\ll p\ll \left(\frac{\log(n)}{n}\right)^{2/(k+3)}$ we have that $\tfrac{\codim(S/I_\Delta)}{\pdim(S/I_\Delta)}\to 1$ in probability, yet the probability that $S/I_\Delta$ is Cohen-Macaulay goes to $0$.
\end{cor}

First we prove a quick lemma bounding the dimension of $\Delta$.
\begin{lemma}\label{lem:dim Delta}
If $p\leq \epsilon$ for some $0<\epsilon<1$ then
$\bP[\dim \Delta \geq \epsilon\cdot n]\to 0$ as $n\to \infty$.
\end{lemma}
\begin{proof}
The dimension of $\Delta$ is the size of the largest $k$-clique in $\Delta$.  Let $N:=\binom{n}{k}$.  The expected number of $k$-cliques in $\Delta$ is $Np^{N}\leq N\epsilon^N$, which goes to zero as $n\to \infty$.
\end{proof}
Note that~\cite[Theorem~1]{bollobas-erdos} provides a much sharper estimate of the dimension of $\Delta$, though we will not need that.
\begin{proof}[Proof of Corollary~\ref{cor:CM and almost CM}]
Lemma~\ref{lem:dim Delta} shows that $\dim \Delta = o(n)$ with high probability.  By Auslander-Buchsbaum, this implies that
\[
n-o(n) \leq \codim(S/I_\Delta) \leq  \pdim(S/I_\Delta) \leq n.
\]
Thus the ratio between $\pdim(S/I_\Delta)$ and $\codim(S/I_\Delta)$ goes to $1$ in probability.

For the statement on Cohen-Macaulayness, using Reisner's Criterion~\cite[Corollary~5.3.9]{bruns-herzog} it suffices to show that there exists a vertex $v\in\Delta$ and an integer $i<\dim\left(\link_{\Delta}(v)\right)$ where $\HH_i\left(\link_{\Delta}(v)\right)\neq 0$. For $\Delta\sim\Delta(n,p)$ and a vertex $v$, the link of $v$ is itself a random flag complex, namely $\link_{\Delta}(v)\sim\Delta(np,p)$.

For convenience we write $m:=np$. In terms of $m$ we can rewrite the left hand side of the original constraints on $p$ as $\frac{1}{m^{2}}\ll p$. For the right hand side of the constraint, since $\frac{1}{n} \ll p$, we have $\log(m)\sim \log(n)$ so we get $p\ll \left(\frac{\log(n)}{n}\right)^{2/(k+3)}\sim \left(\frac{\log(m)}{m}\right)^{2/(k+1)}$. Thus the constraints in terms of $m$ are \[\frac{1}{m^{2}}\ll p \ll \left(\frac{\log(m)}{m}\right)^{2/(k+1)}.\]

For $1\leq t \leq k$, we consider the interval $\frac{1}{m^{2/t}}\ll p \ll \left(\frac{\log m}{m}\right)^{2/(t+1)}$. Since $\frac{1}{m^{2/(t+1)}} \ll \left(\frac{\log(m)}{m}\right)^{2/(t+1)},$ the successive intervals overlap, and it suffices to show that for each of these intervals $\Delta$ is not Cohen-Macaulay with probability approaching $1$.

First let us consider the case where $t\geq2$. Setting $i:=\lfloor t/2\rfloor$ and applying \cite[Theorem~1.1]{kahle} we have $\HH_i\left(\link_{\Delta}(v)\right)\neq 0$ with probability $1-o(1)$. Since $\frac{1}{m^{2/t}}\ll p$, there exist $(t+1)$-cliques and thus $\dim(\link_{\Delta}(v))\geq t$ with probability $1-o(1)$. Together these imply that $\Delta$ is not Cohen-Macaulay with probability $1-o(1)$

We now consider the case $t=k=1$, where we have $\frac{1}{m^2}\ll p \ll \frac{\log m}{m}$.  Thus we apply \cite[Theorem~1]{erdos-renyi-1} to get $\HH_0(\link_{\Delta}(v))\neq 0$ with probability $1-o(1)$. On the other hand, since $\frac{1}{m^2}\ll p$, we have $2$-cliques and thus $\dim(\link_{\Delta}(v))\geq t$ with probability $1-o(1)$
\end{proof}

\end{document}